\documentclass[oneside,12pt]{amsart}

\usepackage[all]{xy}

\usepackage{amscd,amsmath,latexsym,amssymb}
\newtheorem{thm}{Theorem}[section]

\newtheorem{lem}[thm]{Lemma}
\newtheorem{prop}[thm]{Proposition}

\theoremstyle{definition}

\theoremstyle{definition}

\newtheorem{exm}[thm]{Example}

\newtheorem{remark}[thm]{Remark}

\theoremstyle{remark}

\begin{document}
   \title{On the cohomology of polyhedral products with space pairs $(D^1,S^0)$}
   \author[Li~Cai]{Li~Cai}
      \address{Graduate School of Mathematics,
		   Kyushu University, 744, motooka, nishi-ku, Fukuoka-city 819-0395, Japan}
		      \email{l-cai@math.kyushu-u.ac.jp}

    \begin{abstract}
	  In this note we prove that $H(\mathbb{R}\mathcal{Z}_{K})$, 
	  namely the integer cohomology ring of {\em polyhedral products} 
	 with space pairs $(D^1, S^0)$,
	 can also be described explicitly by a multiplicative Hochster's formula. 
    \end{abstract}

    \maketitle

The objective of this note is to understand 
$H(\mathbb{R}\mathcal{Z}_{K})$, in a way parallelly to the work related to
the cohomology of {\em moment-angle complexes} $\mathcal{Z}_{K}$ 
(see e.g. \cite{buchstaber.panov.2}, 
\cite{baskakov3}, \cite{panov}) by Buchstaber, Panov and Baskakov. 

After some constructions, an explicit formula for $H(\mathbb{R}\mathcal{Z}_{K})$
is obtained in Theorem \ref{simplicial}, 
the whole proof of which is completed later after 
Proposition \ref{cup final}. 
Some applications are discussed in Remark \ref{rem}.

\section{On the Additive Structure}
In this note let $\mathbb{Z}$ be the ring of integers and let 
all coefficients be from $\mathbb{Z}$. 

For a non-negative integer $m$, we write $[m]$ for the set $\{1,\cdots , m\}$. 
A simplicial complex $K$ on $[m]$ means the vertice set of $K$ is identified with 
$[m]$, such that 
\begin{enumerate}
  \item there is an unique emptyset which belongs to every simplex of $K$; 
  \item as a subset of $[m]$, $\sigma$ is a simplex of $K$ implies 
	that any face (subset) of $\sigma$ is a simplex.  
	\end{enumerate}
	For instance, if $K$ contains $[m]$ as a simplex,
then every subset $\sigma=\{i_0,\dots ,i_p\}$ is a $p$-simplex of $K$.

The geometric realization of $K$ will be denoted by $|K|$. By 
$\mathbb{R}\mathcal{Z}_{K}$
we denote a specific polyhedral product 
(we follow \cite{buchstaber.panov.2} on this notation) 
defined as the union
\begin{equation}
 \mathbb{R}\mathcal{Z}_{K}=\bigcup_{\sigma\in K}\{(x_1,\cdots, x_m)\in (D^1)^m|
   x_i\in S^0, \text{ when $i\not\in\sigma$}\}, \label{Def}
  \end{equation}
  thus $\mathbb{R}\mathcal{Z}_{K}$ can be embedded in the product 
  $(D^1)^m$ in an obvious way.

  \begin{exm}\label{pentagon}
   It is known that if $K$ is a simplicial $(n-1)$-sphere (a simplicial complex 
   homeomorphic to the $(n-1)$-sphere), then 
   $\mathbb{R}\mathcal{Z}_{K}$ is an $n$-dimensional manifold 
   (see \cite[p. 98]{buchstaber.panov.2}). 
	Let $K$ be the pentagon on $[5]$, namely it is a $1$-dimensional simplicial
	complex with edges 
	$\{i,i+1\}$ for $i=1,2,3,4$ and $\{1,5\}$. For the five $1$-simplices (resp. 
	five $0$-simplices) of $K$, 
   let each $1$-simplex (resp. $0$-simplex) correspond to the $2^3$ squares 
   (resp. the $2^4$ edges)
   $(D^1)^2\times (S^0)^3$ (resp. $D^1\times (S^0)^4$) and 
   let $(S^0)^5$ be the $2^5$ vertices as in definition \eqref{Def}. 
   Then by counting the Euler number we can identify this 
   $\mathbb{R}\mathcal{Z}_{K}$ as a surface with 
   genus $5$.
  \end{exm}

  \subsection{A Cell Decomposition}
  Now we give $(D^1)^m$ a $CW$ structure by a simplicial 
  decomposition to each component $D^1_{i}$ ($i=1,\dots,m$), 
  which descends to $\mathbb{R}\mathcal{Z}_{K}$ as 
  a subcomplex:
  let $\underline{01}_i$ be the $1$-simplex and 
  $0_i$, $1_i$ be two $0$-simplices  
  of the unit interval $[0,1]_{i}\cong D^1_i$ respectively, then each 
  oriented cell of $(D^1)^m$ 
  can be written as 
    \[
	  e_1\times\cdots\times e_m, \text{ $e_i=0_i$, $1_i$  or  $\underline{01}_i$}.
	\]
	With the decomposition above, we find 
	\begin{lem}\label{embedding}
	A cell $e=e_1\times\cdots\times e_m$ of $(D^1)^m$
	belongs to its subspace $\mathbb{R}\mathcal{Z}_{K}$ 
	if and only if	
	\[\sigma_e:=\{i|e_i=\underline{01}_i\}\] 
	is a simplex of $K$. 
  \end{lem}
	In what follows we will abuse the notation $e_i$: it may mean a simplex or the 
	corresponding  simplicial chain, and we will declare what it means depending on 
	different situations. For a topological space $X$, we denote $S(X)$ (resp. $S^*(X)$)
	to be its singular chains (resp. cochains); 
	$C_e(X)$ (resp. $C_e^*(X)$) will mean the cellular chains (resp. cochains) 
	when $X$ is a $CW$ complex, and $C(X)$ (resp. $C^*(X)$) will be denoted for
	the simplicial chains (resp. cochains) when $X$ is a simplicial complex. The
	notation $H(X)$ (resp. $H^*(X)$) means the singular homology (resp. cohomology)
	of $X$.

	\subsection{A Brief Review}\label{cohomology 1}
Let $X=\prod_{i=1}^m|K_i|$ be 
a product space by topologized finite simplicial complexes $K_i$, 
note that now each $C^*(K_i)$ is finitely generated. 
Let ($\otimes_{i=1}^mC^*(K_i);\mathrm{d}$) be 
the differential graded abelian group with generators 
\[
e_1^*\otimes\dots \otimes e_m^* 
	   	  \] 
    where $e_i^*$ is the dual of the simplicial chain generated by $e_i$ in $K_i$,
	and whose differential $\mathrm{d}$ satisfies
	\begin{equation}
	  \mathrm{d}(e_1^*\otimes\dots\otimes e_m^*)=
	  \sum_{i=1}^{m}(-1)^{\sum_{j=1}^{i-1}\mathrm{deg}(e_j^*)}e^*_1\otimes\dots\otimes e_{i-1}^*\otimes
	  \mathrm{d}e_i^*\otimes e_{i+1}^*\otimes\dots\otimes e_m^* \label{d}
    \end{equation}
	on each generator (on the right hand side the differential $\mathrm{d}$ is the usual 
	simplicial coboundary operator).
	For instance, assume that $i=1,2,3$, 
	from $\mathrm{d}0_i^*=-\underline{01}_i^*$, $\mathrm{d}1_i^*=\underline{01}_i^*$ and $\mathrm{d}\underline{01}_i^*=0$, 
	one has
	\[
	  \mathrm{d}(0_1^*\otimes \underline{01}_2^*\otimes 1^*_3)=-\underline{01}_1\otimes\underline{01}_2\otimes 1^*_3-
	  0_1^*\otimes \underline{01}_2^*\otimes \underline{01}^*_3.
    \]
	 
	$C_e^*(X)$ means the set of cellular cochains generated by dual cells
	\[
	  (e_1\times\dots\times e_m)^*,
	  \]
	on which we define a morphism 
	\begin{equation}
	  \begin{CD}
		f:C^*(K_1)\otimes\dots\otimes C^*(K_m) @>>> C_e^*(X)\\
		e_1^*\otimes\dots\otimes e_m^* @>>> (e_1\times\dots\times e_m)^*,
	  \end{CD}
	  \label{chain map}
	\end{equation}
	which is a $1$-$1$ correspondence by a comparison of basis, thus 
	we can equip $C_{e}^*(X)$  with 
	a differential $\mathrm{d}$ such that $f$ is a cochain map, 
	and it will be denoted as $(C_{e}^*(X);\mathrm{d})$.
		\begin{prop}\label{cohomology}
	  Consider $X=\prod_{i=1}^m|K_i|$ as a $CW$ complex with 
	  cellular cochains $(C_e^*(X);\mathrm{d})$,
	  and assume that $A$ is a subcomplex of $X$ subject 
	  to this cell decomposition, 
	  then there is an isomprphism
	  \[
		H^*(A)\cong H(C_e^*(A);\mathrm{d})\]
	preserving degrees. For instance, we have isomorphisms
		\[
		  \begin{CD} 
			\displaystyle{H(\bigoplus_{\sum_{i=1}^{m}p_i=p}C^{p_1}(K_1)\otimes\dots\otimes C^{p_m}(K_m);\mathrm{d})}
			@>f>\cong>H^p(C^*_e(X);\mathrm{d})\cong H^p(X)
		 \end{CD}
		   \]
	    where $p_i$ are non-negative integers. 
	  \end{prop}
    A standard cellular (co)homology argument (see e.g. \cite[Theorem 57.1, p. 339]{munkres}) will
    provide a proof without involving cup products.	
	  We will give full version of this proposition later (i.e. Proposition \ref{cup final}),
	  but here the simplicial cochain algebra of $C^*(D^1_i)$ with simplicial cup products
	  will be presented, because it motivates the 
	  algebraic construction later. Henceforth let us work with the following basis
		\begin{equation}
		  \{\mathfrak{1}_i=1^*_i+0_i^*, t_i=1_i^*, u_i=\underline{01}_i\}\label{new basis}
	  \end{equation}
		of $C^*(D^1_i)$ rather than the original $\{0_i^*, 1_i^*, \underline{01}_i^*\}$. By
		definition (see e.g. \cite[p. 292]{munkres}), we have the following relations on simplicial 
		cup products:
		\[\mathfrak{1}_i\cup t_i=t_i\cup \mathfrak{1}_i,\quad \mathfrak{1}_i\cup u_i=u_i\cup \mathfrak{1}_i, \quad
          t_i\cup t_i=t_i, \quad u_i\cup u_i=0, \quad t_i\cup u_i=0, \quad u_i\cup t_i=u_i.
		  \]

		\subsection{An Algebraic Description}
	In what follows, $x_{\sigma}$ will mean the monomial $x_{i_0}\dots x_{i_p}$
    when $\sigma=\{i_0, \dots ,i_p\}$ is a subset of $[m]$. 
	By $\mathbb{Z}[u_1,\dots ,u_m;t_1,\dots ,t_m]$ we denote the differential graded
	free $\mathbb{Z}$-algebra with $2m$ generators, such that 
   \begin{equation}
   \mathrm{deg}u_i=1, \quad
   \mathrm{deg}t_i=0; \quad
   \mathrm{d}u_i=0, \quad \mathrm{d}t_i=u_i,
   \label{deg. rel.}
   \end{equation}
   where the differential $\mathrm{d}$ follows the rule \eqref{d} on each 
   monomial.
   Let $R$ be the quotient algebra of $\mathbb{Z}[u_1,\dots ,u_m;t_1,\dots ,t_m]$ subject 
   to the following relations
   \begin{equation}
	 u_it_i=u_i,\quad t_iu_i=0,\quad u_it_j=t_ju_i, 	 
	 \quad t_it_i=t_i, \quad u_iu_i=0,
	 \quad u_iu_j=-u_ju_i, \quad t_it_j=t_jt_i, \label{relation}
	 \end{equation}
   for $i,j=1,\dots,m$ such that $i\not=j$, in which the Stanley-Reisner ideal 
   $\mathcal{I}_K$ is defined to be generated by all square-free 
	monomials $u_\sigma$ such that $\sigma$ is not a simplex of $K$. 

	One can check that the relations \eqref{relation} listed above are compatible with 
	the differential $\mathrm{d}$, for instance, 
	\[
	  \mathrm{d}(t_it_i)=u_it_i+t_iu_t=u_it_i=u_i=\mathrm{d}t_i.
	  \]
		
	\begin{thm}\label{additive part}
	  There is a (degree-preserving) ring isomorphism
	  \begin{equation}
		H(R/\mathcal{I}_K;\mathrm{d})
		\cong H(\mathbb{R}\mathcal{Z}_{K}).
		\label{add. iso.}
	  \end{equation}
	\end{thm}
	\begin{proof}
	  	  Here we only prove the additive part of this theorem, and the whole 
		  proof involving products will be completed later after  
		  Proposition \ref{cup final}. 

		  We start from a cochain map, which is generated by mapping 
	  each square-free monomial to a tensor product of simplicial cochains, 
	  namely
	  \[
		\begin{CD}
		\theta:R @>>> C^*(D^1_i)\otimes\dots\otimes C^*(D^1_m)\\
         x_1\dots x_m @>>> e_1^*\otimes\dots\otimes e_m^*,
 \end{CD}
 \]
 in which (we are using the basis \eqref{new basis}) 
 \[e_i^*=\begin{cases}\mathfrak{1}_i, & \text{if $x_i=1$},\\
   t_i,   & \text{if $x_i=t_i$},\\
   u_i,   & \text{if $x_i=u_i$}.
 \end{cases}\]
 The additive generators of $R$ are square-free monomials 
 follows from the relations \eqref{relation}, and by definition 
 $\theta$ induces an isomorphism between abelian groups such that 
 $\mathrm{d}\circ\theta=\theta\circ\mathrm{d}$.
 Then from the following diagram about cochain maps (in which $\mathrm{q}$ is the quotient 
 and $\mathrm{i}^*$ is induced by the inclusion)
  \begin{equation}
		\begin{CD}
	    R		@>f\circ\theta>\cong> C^*_e(D^1_1\times\dots\times D^1_m)   \\
		@V \mathrm{q} VV                                                           @V\mathrm{i}^*VV                                  \\
		R/\mathcal{I}_{K}@>>\cong>
	 C^*_e(\mathbb{R}\mathcal{Z}_{K}),
   \end{CD}\label{correspondence}
	  \end{equation}
together with Lemma \ref{embedding}, we observe that the quotient part 
in the Stanley-Reisner ideal corresponds to the cells not contained in 
$\mathbb{R}\mathcal{Z}_{K}$, thus 
by Proposition \ref{cohomology}, this part of proof is completed.
	\end{proof}
	\subsection{A Combinatorial Description (Hochster's Formula)}\label{Hochster}
	With the notation $x_\sigma$, we can write each square-free 
	monomial of $R/\mathcal{I}_{K}$ in the form 
	\[u_{\sigma}t_{\tau}, \quad \sigma\cap\tau=\emptyset,\quad \sigma\in K \]
	(for instance, $u_2t_3u_4t_5t_6=u_{2,4}t_{3,5,6}$). 
    
	When $\omega$ is a subset of $[m]$ (may be empty), we denote $K_{\omega}$ to be 
	the full subcomplex of $K$ in $\omega$, namely it is the subcomplex
	\[
	  \{\sigma\cap \omega|\sigma\in K\},
	  \]
	  and $R/\mathcal{I}_{K}|_{\omega}$ will be denoted as a subgroup of $R$ (as a group) 
	  generated by 
	\[
	  \{u_{\sigma}t_{\tau}| \sigma\in K \ 
	   s.t. \ \sigma\sqcup\tau=\omega, \ \sigma\cap\tau=\emptyset\}.
	  \]
	  We observe that $R/\mathcal{I}_{K}|_{\omega}$ is closed 
	  under the differential $\mathrm{d}$ and 
       $R/\mathcal{I}_{K}=\bigoplus_{\omega\subset [m]}R/\mathcal{I}_{K}|_{\omega}$, therefore
	\[
	  H(R/\mathcal{I}_{K};\mathrm{d})\cong
	  \bigoplus_{\omega\subset [m]}H(R/\mathcal{I}_{K}|_{\omega};\mathrm{d}).
	  \]
	Each $\omega\subset [m]$ yields 
	a $1$-$1$ cochain map (on the right hand side we use the usual 
	simplicial coboundary operator)
	\begin{equation}
	  \begin{CD}
		\lambda_{\omega}: R/\mathcal{I}_{K}|_{\omega} @>\cong>> C^*(K_{\omega})\\
		                 u_{\sigma}t_{\tau}@>>> \sigma^*,
					   \end{CD}\label{lambda}
	  \end{equation}
	  which induces an additive isomorphism of cohomology
	  \[
		\begin{CD}
		  H^p(R/\mathcal{I}_{K}|_{\omega})@>\lambda_{\omega}>\cong> \widetilde{H}^{p-1}(K_{\omega}),
	  \end{CD}
		\]
		because 
		\[\mathrm{d}(t_1\dots t_m)=\sum_{i=1}^{m}t_1\dots t_{i-1}u_it_{i+1}\dots t_m\]
		behaves as the augmentation and $\lambda_{\emptyset}^*$ maps $H^0(R)$ 
		isomorphically onto $\widetilde{H}^{-1}(\emptyset)$.

		\begin{remark}
    Indeed, $\lambda$ induces cochain maps since every time we jump over
	an index in $\sigma$, we will get a $-1$. 
	Here the construction of $\lambda$ is motivated by Baskakov's idea in \cite{baskakov}.
		\end{remark}
      
		Therefore we can equip $\bigoplus_{\omega\subset [m]}\widetilde{H}^{p-1}(K_{\omega})$ with 
	  a product structure described by simplicial cochains, such that cochain map $\lambda$ induces an 
	  isomorphism between cochain algebras.
	  Together with Theorem \ref{additive part}, we have the following theorem,
	  the full proof of which will be completed later. 

	  \begin{thm}\label{simplicial}
		There are isomorphisms
		\[
		\begin{CD}
		  \displaystyle{\bigoplus_{\omega\subset [m]}\widetilde{H}^{p-1}(K_{\omega})}@>\cong> \lambda^{-1}>
		  H^p(R/\mathcal{I}_{K})@>\cong>> H^p(\mathbb{R}\mathcal{Z}_{K})
		\end{CD}
		\]
	  for all positive integers $p$. Moreover, after endowing 
	  $\bigoplus_{\omega\subset [m]}\widetilde{H}^{p-1}(K_{\omega})$ the product structure induced by $\lambda$, 
	  they are isomorphisms between $\mathbb{Z}$-algebras.
	  \end{thm}
	   Thus the cohomology ring of $\mathbb{R}\mathcal{Z}_{K}$ will be clear once we 
	   write down all cochains representing the (reduced) simplicial cohomology 
	   for each $K_{\omega}$.
	  \begin{exm}Assume that $K$ is the pentagon described in Example \ref{pentagon}, let 
		us compute the cohomology of $\mathbb{R}\mathcal{Z}_{K}$.
		First we can write down all the full subcomplexes of $K$ with non-trivial cohomology, namely
		\begin{align*}
		  1:&\widetilde{H}^{-1}(K_{\emptyset})|1; \\
		  \alpha_1-\alpha_5:&\widetilde{H}^0(K_{1,3})|u_1t_3, \widetilde{H}^0(K_{1,4})|u_1t_4, 
          \widetilde{H}^0(K_{2,4})|u_2t_4, \widetilde{H}^0(K_{2,5})|u_2t_5, 
		  \widetilde{H}^0(K_{3,5})|u_3t_5;\\
		  \beta_1-\beta_5:&\widetilde{H}^0(K_{1,3,4})|u_1t_{3,4}, \widetilde{H}^0(K_{2,4,5})|u_2t_{4,5}, 
		  \widetilde{H}^0(K_{3,5,1})|u_3t_{1,5}, \widetilde{H}^0(K_{4,1,2})|u_4t_{1,2}, 
		  \widetilde{H}^0(K_{5,2,3})|u_5t_{2,3};\\
		  \gamma:&\widetilde{H}^1(K_{1,2,3,4,5})|u_{1,2}t_{3,4,5},
		  \end{align*}
           in which we write a generator for each group, via $\lambda$. Thus (by Theorem \ref{simplicial}) 
           $H(\mathbb{R}\mathcal{Z}_{K})$ has ten generators in degree $1$ 
		   (i.e. $\alpha_i$, $\beta_i$, $i=1,\dots ,5$), and one 
		   generator in degree $0$ (i.e. $1$ the unit) and degree $2$ 
		   (i.e. $\gamma$) respectively.            
		   Note that these generators may not be unique: for instance,
		 \[
		   \mathrm{d}(u_1t_{2,3,4,5})=-u_{1,2}t_{3,4,5}-u_{1,5}t_{2,3,4}+
		   \text{terms equivalent to $0$},
		   \]
		   and in the same way $-u_1t_{3,4}$ is equivalent to $u_3t_{1,4}+u_4t_{1,3}$.
		   A little computation shows, up to equivalences in cohomology, all non-trivial cup products are listed
		   as follows (where we omit the notation $\cup$)
		   \[\gamma=\alpha_1\beta_2=\alpha_2\beta_5=\alpha_3\beta_3=\alpha_4\beta_1=
			 \alpha_5\beta_4=\beta_1\beta_2=\beta_2\beta_3=\beta_3\beta_4=\beta_4\beta_5=\beta_5\beta_1.
			 \]
		   We can use the construction 
		   \eqref{correspondence} 
		   in the proof of Theorem \ref{additive part} 
		   to identify all corresponding cellular cochains. 
		\end{exm}
		\begin{remark}\label{rem}
		From Theorem \ref{simplicial}, after a comparison with 
		the result in \cite{buchstaber.panov.2}, it follows that 
		there is an additive isomprphism between 
		 $H(\mathcal{Z}_{K})$ and $H(\mathbb{R}\mathcal{Z}_{K})$
		 if we forget about degrees.
        For instance, we can also 
		find arbitrary torsions in $H(\mathbb{R}\mathcal{Z}_{K})$ 
		when $K$ are special simplicial spheres (see \cite[Theorem 11.11]{BM04}).  
        As algebras, the difference between 
		$H(\mathcal{Z}_{K})$ and $H(\mathbb{R}\mathcal{Z}_{K})$
		comes from defferent sign rules and square rules.  
 	  \end{remark}
		\section{On the Cup Product}
		In this section we will describe cup products in Proposition \ref{cohomology}, 
		and then finish the proofs of Theorem \ref{additive part} and 
	Theorem \ref{simplicial}. 

    Assume that $X$ and $Y$ are two topological spaces, we define a map 
	$ g:S^*(X)\otimes S^*(Y)\to \mathrm{Hom}(S(X)\otimes S(Y),\mathbb{Z})$
	 generated by 
	  \begin{equation}
		g(c^*_X\otimes c^*_Y)(c'_X\otimes c'_Y)= c^*_X(c'_X) c^*_Y(c'_Y), 
		\quad c^*_X\in S^*(X), \ \ c^*_Y\in S^*(Y),\label{define g}
	  \end{equation}
	  (the evaluation is trivial when degrees do not match) 
	  for each singular trangles $c'_X$ and $c'_Y$.

	By Eilenberg-Zilber theorem, there is a 
	natural chain equivalence $\mu$ between singular chain 
	complexes (unique up to chain homotopy, see \cite{eilenberg. maclane. 2}; 
	the defferential on tensors follows the rule \eqref{d})
	\[
	\begin{CD}
	  S(X\times Y)@>\mu_{X\times Y}>\cong> S(X)\otimes S(Y),
  \end{CD}
  \]	
  and the cup product is defined 
  via the following 
\[
  \begin{CD}
	S^*(X)\otimes S^*(X)@>\mu^*_{X\times X}\circ g >> S^*(X\times X)@>d^*>> S^*(X)
  \end{CD}
  \]
  in which $d^*$ is induced by the diagonal map $X\to X\times X$ (the {\em cross product} 
  is defined by $\mu^*_{X\times X}\circ g$).
 
  From now on let $\mathrm{Top}^m$ be the category with 
  $m$-products of toplogical spaces as objects, whose morphisms
  are $m$-products of continuous maps, and let $\mathcal{C}$ be 
  the category of chain complexes equiped with chain maps.

  Consider $X=\prod_{i=1}^m|K_i|$ as an object of 
  $\mathrm{Top}^m$: a product by finite topologized 
  simplicial complexes.

  Thus in the following diagram, all morphisms can be considered
  as natural transformations between functors from 
  $\mathrm{Top}^m$ to $\mathcal{C}$, where $d_i$ is the diagonal chain map 
  from $S(|K_i|)$ to $S(|K_i|)\otimes S(|K_i|)$, and $\widetilde{\mu}$ 
  is defined by using $\mu$ for $m-1$ times, 
  factoring out one component each time:
  \begin{equation}
   \xymatrix{	
	 &  S(\prod_{i=1}^m|K_i|)\otimes S(\prod_{i=1}^m|K_i|)
	 \ar[r]^{\widetilde{\mu}_{X}\otimes\widetilde{\mu}_{X}} &
	 \otimes_{i=1}^{m}S(|K_i|)\otimes \otimes_{i=1}^{m}S(|K_i|)  \\
	 S(\prod_{i=1}^{m}|K_i|) \ar[ru]^{\mu_{X\times X}\circ d} \ar[rd]^{\widetilde{\mu}_{X}}   \\
	 & S(|K_1|)\otimes\dots\otimes S(|K_m|)
	 \ar[r]^{\otimes_{i=1}^{m}\mu_{|K_i|\times |K_i|}\circ d_i}      
	&  \otimes_{i=1}^{m}S(|K_i|)\otimes S(|K_i|)	\ar[uu]^T 	,	
  }\label{chain homotopy}
	 \end{equation}
	in which $T$ is defined as follows. 
	For each tensor by singular triangles 
	$\otimes_{i}c_i\otimes c'_i\in\otimes_{i=1}^{m}S(|K_i|)\otimes S(|K_i|)$, 
  \[
	T(c_1\otimes c'_1\otimes\dots\otimes c_m\otimes c'_m)=
	(-1)^{\varepsilon(c,c')}c_1\otimes\dots\otimes c_m\otimes c'_1\otimes\dots\otimes c'_m	\]
	where the $\mathrm{mod}$ $2$ invariant 
		$\varepsilon(c,c')$ is 
	generated by this rule: each operation of changing the order of 
	any adjacent two terms will produce a multiplication 
	of their degrees, thus we obtain $\varepsilon(c,c')$ by 
	summing up these multiplications 
	for an arbitrary sequence of operations, for instance one 
	choice for $\varepsilon(c,c')$ can be 
	\[
	  \varepsilon(c,c')=\sum_{i=1}^{m}\mathrm{deg}c'_i\sum_{j>i}\mathrm{deg}c_j.
	  \]
	It is straightforward to check that $T$ is a chain map. 
	From the method of {\em acyclic model} with 
	the model category as $m$-products of 
	topologized simplices (see \cite{eilenberg. maclane. 1}, 
	the case for $2$-products can be found in \cite[p. 252]{spanier}),  
	it follows that there is a natural 
	chain homotopy between 
	\[
	  \widetilde{\mu}_X\otimes\widetilde{\mu}_X\circ \mu_{X\times X}\circ d 
	  \text{\quad and \quad } T\circ 
	  (\otimes_{i=1}^{m}\mu_{|K_i|\times |K_i|}\circ d_i)\circ\widetilde{\mu}_X.
	  \] 
	Therefore after taking $\mathrm{Hom}$ to diagram \eqref{chain homotopy}, we obtain the following 
	(where $g_1$, $g_2$ are defined from \eqref{define g}) 
	\[\begin{CD}
		\otimes_{i=1}^{m}C^*(K_i)\otimes \otimes_{i=1}^{m}C^*(K_i)@>
		\widetilde{T}^*:=g^{-1}_2 \circ T^*\circ g_1 >>
	  \otimes_{i=1}^{m}C^*(K_i)\otimes C^*(K_i)\\
       @V g_1 VV                                 @V g_2 VV   \\
	   \mathrm{Hom}(\otimes_{i=1}^{m}S(K_i)\otimes \otimes_{i=1}^{m}S(K_i),\mathbb{Z})@>T^*>>
	   \mathrm{Hom}(\otimes_{i=1}^{m}S(K_i)\otimes S(K_i),\mathbb{Z}),
	\end{CD}
	  \]
	  in which $\widetilde{T}^*$ is well defined since 
	  $\mathrm{Im}(T^*\circ g_1)\subset \mathrm{Im}(g_2)$, 
	  and here every $C^*(K_i)$ is considered as a subgroup of $S^*(|K_i|)$ ($i=1,\dots ,m$), 
	  in which only simplicial chains are non-trivially evaluated.

	  \begin{prop}\label{cup final}
		Assume that $e^*=\otimes_{i=1}^me_i^*$ and ${e^*}'=\otimes_{i=1}^{m}{e^*}_i'$ are tensors of 
		dual simplices in $(\otimes_{i=1}^mC^*(K_i)\subset)\otimes_{i=1}^mS^*(|K_i|)$, 
		then as elements in $S^*(\prod_{i=1}^{m}|K_i|)$ on both sides, we have  	
		\begin{equation}
		  \widetilde{\mu}^{*}_{X}\circ g (\otimes_{i=1}^me_i^*)\cup
		  \widetilde{\mu}^{*}_{X}\circ g (\otimes_{i=1}^{m}{e^*}_i')=
		  (-1)^{\varepsilon(e^*,{e^*}')}\widetilde{\mu}^{*}_{X}\circ g
	  (\otimes_{i=1}^{m}e_i^*\cup {e^*}_i'),
	  \label{shuffle}
	  \end{equation}
	  up to a cochain homotopy. Moreover, 
	  consider $X=\prod_{i=1}^{m}|K_i|$ as a $CW$ complex whose cells are $m$-products
	  of (topologized) simplices with $i$-th component in $K_i$, and suppose that $A$ is a subcomplex of $X$ with 
	  $e^*$ and ${e^*}'$ in $C^*_e(A)$, then as elements in $S^*(A)$, 
	  \eqref{shuffle} holds up to a cochain homotopy.
	\end{prop}
	\begin{proof}
	  After a little chasing
	  on the diagram \eqref{chain homotopy}, the first statement follows from
	  \begin{align*} 
		\widetilde{\mu}^{*}_{X}\circ g (\otimes_{i=1}^me_i^*)\cup
		  \widetilde{\mu}^{*}_{X}\circ g (\otimes_{i=1}^{m}{e^*}_i')&=
     d^* \circ \mu^*_{X\times X}\circ\widetilde{\mu}^*_X\otimes\widetilde{\mu}^*_X\circ g_1(\otimes_{i=1}^me_i^*\otimes\otimes_{i=1}^{m}{e^*}_i')\\
	 &=\widetilde{\mu}^*_X\circ(\otimes_{i=1}^{m}\mu_{|K_i|\times |K_i|}\circ d_i)^*
	 \circ T^*\circ g_1(\otimes_{i=1}^me_i^*\otimes\otimes_{i=1}^{m}{e^*}_i')\\
	 &=(-1)^{\varepsilon(e^*,{e^*}')}\widetilde{\mu}^*_X\circ(\otimes_{i=1}^{m}\mu_{|K_i|\times |K_i|}\circ d_i)^*
	 \circ g_2(\otimes_{i=1}^me_i^*\otimes{e^*}_i')\\
	 &=(-1)^{\varepsilon(e^*,{e^*}')}\widetilde{\mu}^{*}_{X}\circ g
	  (\otimes_{i=1}^{m}e_i^*\cup {e^*}_i'),
   \end{align*}
   in which from the second line to the third we use the (co)chain homotopy 
   and change $\widetilde{T}^*$ for the sign $(-1)^{\varepsilon(e^*,{e^*}')}$. 
	  Now we prove the part about the subcomplex $A$.
	  Denote all cells of $A$ by $E_{A}$,
	  we can consider 
	  $E_{A}$ as a subcategory of $\mathrm{Top}^m$ with inclusions as
	  morphisms, then there is a chain equivalence (see e.g. \cite[Theorem 31.5]{munkres}, together 
	  with the fact that each $e$ in $E_{A}$ is a {\em neighborhood 
	  deformation retract} of $A=\bigcup_{e\in E_{A}}e$) 
	  \[
		S(A)\cong\mathrm{colim}_{e\in E_{A}}S(\prod_{i=1}^{m}e_i)
		\]
		where $e=\prod_{i=1}^{m}e_i$ with each simplex $e_i$ in $K_i$. Next, from the 
	  naturality of $\mu$ we have the diagram
	  \[
		\begin{CD}
		  S(X)@>\widetilde{\mu}_{X}>> \otimes_{i=1}^{m}S(|K_i|)\\
           @AAA                       @AAA \\
		   S(A)\cong\mathrm{colim}_{e\in E_{A}}S(\prod_{i=1}^{m}e_i)
		   @>\widetilde{\mu}_{X}|_{A} >>\mathrm{colim}_{e\in E_{A}}\otimes_{i=1}^{m}S(e_i),  
		\end{CD}
		\]
		to which after taking $\mathrm{Hom}$, it is not difficult to observe
		\[
	  \begin{CD}
		C_e^*(X)\cong (\otimes_{i=1}^{m}C^*(K_i))
		@>\widetilde{\mu}^*_{X}\circ g >> S^*(X)\\
     @VVV                        @VVV                     \\    
	 C_e^*(A) @>\widetilde{\mu}^*_{X}|_A\circ g_{A} >> S^*(A).
	\end{CD}
	  \]
	  where $g_A:C_e^*(A)\to \mathrm{Hom}
	  (\mathrm{colim}_{e\in E_{A}}\otimes_{i=1}^{m}S(e_i),\mathbb{Z})$ 
	  is defined by evaluating 
	  each $\tilde{e}^*\in C^*_{e}(A)$ non-trivially on its dual 
	  $\tilde{e}\in\mathrm{colim}_{e\in E_{A}}\otimes_{i=1}^{m}S(e_i)$ only. Then we can 
	  replace everything of diagram \eqref{chain homotopy} by colimits, and repeat
	  the process for $X$.
	\end{proof}
	It is not difficult to find the sign rule for commutativity defined 
	by $\varepsilon$ coincides with the one in $R/\mathcal{I}_{K}$, 
	which is defined in \eqref{relation}. Therefore we have completed all proofs.

	\bibliographystyle{amsalpha}

\begin{thebibliography}{99}



	  \bibitem{baskakov} I.~Baskakov, {\em Cohomology of K-powers of spaces and the
		combinatorics of simplicial divisions}, Russian Math. Surveys 57
		(2002), no. 5, 9899990.


	  \bibitem{baskakov3}   I.~Baskakov, V.~Buchstaber and T.~Panov, {\em  Algebras of cellular cochains,
		and torus actions}, Russian Math. Surveys 59 (2004), no. 3, 5625563.


	  \bibitem{buchstaber.panov.2} V.~Buchstaber and T.~Panov, {\em  Torus actions and their applications in topology and
		combinatorics }, AMS University Lecture Series, volume 24, (2002).



	  \bibitem{panov} T.~Panov, {\em Cohomology of face rings and torus actions}, arXiv:math.AT/0506526.

	  \bibitem{BM04} 
		Fr\'{e}d\'{e}ric Bosio and Laurent Meersseman,
		{\em Real quadrics in $\mathbb{C}^n$, complex manifolds and convex
		polytopes}, Acta Mathematica \textbf{197} (2006), no.~1, 53-127. 
		\MR{2285318}
       
	  \bibitem{munkres} J.~Munkres, {\em Elements of algebraic topology}, the Benjamin/Cummings Publishing Company,
      California, 1984.

      \bibitem{eilenberg. maclane. 1} S.~Eilenberg and S.~MacLane, {\em Acyclic models}, American Journal of
         Mathematics, vol. 75 (1953), pp. 189-199.

	   \bibitem{eilenberg. maclane. 2} \bysame, {\em On the groups $H(\prod,n)$, $\uppercase\expandafter{\romannumeral1}$}, Annals of
         Mathematics, vol. 88, No. 1, (1953), pp. 55-106.

	   \bibitem{spanier}E.~Spanier, {\em Algebraic topology}, Springer-Verlag, New York, 1966. 















		


	\end{thebibliography}
	
	\end{document}